\documentclass[ejs,preprint]{imsart}
    \setattribute{journal}{name}{}
\RequirePackage[OT1]{fontenc}
\RequirePackage{amsthm,amsmath,amsfonts,amssymb,multirow,url,graphicx,enumerate,lscape}
\RequirePackage[numbers]{natbib}
\RequirePackage[colorlinks,citecolor=blue,urlcolor=blue]{hyperref}

\arxiv{arXiv:0000.0000}

\startlocaldefs
\numberwithin{equation}{section}
\theoremstyle{plain}

\endlocaldefs

\newtheorem{lem}{Lemma}

\newtheorem{prop}{Proposition}
\newtheorem{theorem}{Theorem}
\newtheorem{rem}{Remark}

\def\tr{{\textrm tr}}

\begin{document}

\begin{frontmatter}
\title{On the theoretic and practical merits of the banding estimator for large covariance matrices}
\runtitle{The Banding Estimator }

\begin{aug}
\author{\fnms{Luo} \snm{Xiao}\ead[label=e1]{lxiao@jhsph.edu}}
\address{Department of Biostatistics\\
Bloomberg School of Public Health\\
Johns Hopkins University\\
Baltimore, Maryland\\
\printead{e1}}

\and
\author{\fnms{Florentina} \snm{Bunea}\ead[label=e2]{fb238@cornell.edu}}
\address{Department of Statistical Science\\
Cornell University\\
Ithaca, New York\\
\printead{e2}}
\runauthor{L. Xiao and F. Bunea}

\end{aug}

\begin{abstract}
This paper considers the banding estimator  proposed in
\citep{Bickel:08b} for  estimation of large  covariance matrices. We
prove that the banding estimator achieves rate-optimality under the
operator norm, for a class of approximately banded covariance
matrices, improving the existing results in \cite{Bickel:08b}. In
addition, we propose a Stein's unbiased risk estimate
(Sure)-type approach for selecting the bandwidth for the banding estimator. Simulations indicate that the Sure-tuned banding estimator outperforms competing estimators. 
\end{abstract}


\end{frontmatter}

\section{Introduction}
High dimensional covariance estimation has attracted a  lot of attention in recent years. This was largely motivated by the fact that the sample covariance 
matrix $\hat{\Sigma}$, based on a sample of size $n$, may  not necessarily  be a consistent estimator of the covariance matrix $\Sigma$ of a random vector $X\in \mathbb{R}^p$, if $p>n$. In particular, it is well known that in spike covariance models the eigenvalues of the sample covariance are inconsistent estimators of their population counterparts \cite{ Baik:06, Johnstone:01}.
For high dimensional population covariance matrices with low
dimensional structures, consistent estimators can be obtained,
depending on the nature of the low dimensional structure, by banding
\cite{Bickel:08b}, tapering \cite{Bickel:08b,Cai:10, Cai:12, Furrer:07,  Yi:13}, 
and thresholding \cite{Bickel:08a, Cai:11,  Karoui:08}. Moreover, some sparse estimators ensure positive
definiteness through the choice of objective function \cite{Bien:11, Rothman:12} or by the addition of an explicit constraint
on the smallest eigenvalue \cite{Bien:11,  Liu:13, Xue:12}. 
Cholesky-decomposition based regularization has also been intensively studied
 \cite{Huang:06, Lam:09, Rothman:08, Wu:03}. Besides estimation, various tests have been proposed for examining the postulated low complexity structure.  Since our work focuses on  estimation of approximately banded matrices, we only mention tests relevant to such structures, developed, among others,  by  \cite{Cai:11b,  Ledoit:02, Chen:10, Jiang:04, Liu:08,  Qiu:12, Zhang:13}.

In this paper we re-visit the banding estimator in \cite{Bickel:08b} and address the following important open question: Does  the banding estimator achieve the operator norm optimal rate derived in \cite{Cai:10} over the following class of covariance matrices introduced by \cite{Bickel:08b}?
\begin{equation*}
\begin{split}
\mathcal{C}_{\alpha} &= \mathcal{C}_{\alpha}(M_0,M_1)\\
&:= \biggl \{\Sigma = (\sigma_{ij})_{1\leq i, j\leq p}: \max_j \sum_{|i-j|\geq k} |\sigma_{ij}| \leq M_1k^{-\alpha} \text{ for all } k>0,\\
&\qquad\quad \text{ and } 0 <M_0^{-1}\leq \lambda_{\min}(\Sigma), \lambda_{\max}(\Sigma) \leq M_0\biggr \}.
\end{split}
\end{equation*}
The class $\mathcal{C}_{\alpha}$ will be referred to as the class of approximately banded covariance matrices. 

Assume $X_k = (X_{k,1},\dots, X_{k,p})^T, k = 1,\dots, n$,  are i.i.d.\ realizations of $X\sim \text{MVN}(\mu,\Sigma)$ with $\Sigma \in \mathcal{C}_{\alpha}$. Let $\hat\Sigma = (\hat\sigma_{ij})_{1\leq i, j\leq p}$ be the sample covariance matrix, i.e., $\hat\sigma_{ij} = (n-1)^{-1}\sum_k (X_{k,i} - \bar X_i)(X_{k,j} - \bar X_j)$,
where $\bar X_i = n^{-1}\sum_k X_{k,i}$. The banding estimator is defined as
\begin{equation}
\label{banding}
\hat\Sigma_K = (\hat\sigma_{ij} 1_{|i-j|\leq K-1})_{1\leq i, j\leq p}
\end{equation}
and $K$ is referred to as the bandwidth of $\hat \Sigma_K$.

It is shown in \cite{Cai:10}  that  the banding estimator  is rate
optimal under the Frobenius norm and that the operator-norm rate
derived in \cite{Bickel:08b} is sub-optimal. However,  it remains
unclear whether or not the banding estimator can be rate-optimal under
the operator norm. To date, two operator-norm minimax rate-optimal
estimators have been proposed: the tapering estimator \cite{Cai:10}
and a block-thresholding estimator \cite{Cai:12b}, the latter also
being minimax adaptive.  \cite{Yi:13} proposed a Stein's unbiased risk
estimation (Sure)-type approach for selecting the bandwidth of the
tapering estimator, but the resulting estimator is aimed at minimizing the Frobenius risk instead of the operator-norm risk.  The block-thresholding estimator, while minimax adaptive, is found in our simulations to have inferior finite-sample performance compared to other estimators.

The discussion above motivates the work presented in this paper.  First,  we provide a proof for establishing the rate optimality of  the banding estimator under the operator norm, thus improving the rate in \cite{Bickel:08b} and filling the existing theoretic gap. Second,  we provide a practical approach for selecting the bandwidth for the banding estimator  by a novel approach inspired by the Stein's unbiased risk estimate (Sure)  \citep{Stein:81}. We demonstrate in simulations that the resulting banding estimator outperforms other competing estimators.

The remainder of the paper is organized as follows. In Section \ref{sec:rate} we state our main theoretic result. In Section \ref{sec:sure} we consider bandwidth selection. In Section \ref{sec:sim} we conduct simulations to compare the proposed estimator with other competing estimators. In Section \ref{sec:derivation} we provide a detailed proof of  the result in Section \ref{sec:rate}.

\section{The banding estimator is operator-norm rate optimal}\label{sec:rate}
In this section we show that the banding estimator $\hat\Sigma_K$
defined in~(\ref{banding}) is operator-norm rate optimal over
$\mathcal{C}_{\alpha}$. We use the following notation: Let $a \lesssim
b$ denote an inequality that holds up to a multiplicative constant;
let $a_n \asymp b_n$ denote that there exist two constants $c$ and $C$
such that $c a_n \leq b_n\leq C a_n$ for large $n$; finally, for an arbitrary matrix $A= (a_{ij})_{ij}$, define $A^{abs} = (|a_{ij}|)_{ij}$.  For $\Sigma\in \mathcal{C}_{\alpha}$, it is easy to show that
$\|\Sigma^{abs}\|_{op}$ is bounded by $M_0 + M_1$.

\begin{theorem}
\label{thm}
For the banding estimator $\hat\Sigma_K$ with $\Sigma\in \mathcal{C}_{\alpha}$, 
there exists a constant $c>0$ such that
$$
\mathbb{P}\left\{ \|\hat\Sigma_K - \Sigma\|_{op} \geq  c K^{-\alpha} + c \|\Sigma^{abs}\|_{op}\sqrt{\frac{K + \log p}{n}}\right\}\lesssim p^{-1}, \,\, \text{ for any } K\geq 1.
$$
Furthermore,
$$
\mathbb{E} \|\hat \Sigma_K - \Sigma\|_{op}^2 \lesssim K^{-2\alpha} + \frac{K +\log p}{n}, \,\, \text{ for any } K\geq 1.
$$
\end{theorem}
\begin{rem}
If $K \asymp n^{1/(2\alpha+1)}$, then
$$
\mathbb{E} \|\hat \Sigma_K - \Sigma\|_{op}^2 \lesssim n^{-2\alpha/(2\alpha+1)} + \frac{\log p}{n},
$$
which is the optimal rate  over $\mathcal{C}_{\alpha}$ under the operator norm \citep{Cai:10}.

\end{rem}

We explain below the difference between the derivation in \cite{Bickel:08b}, which leads to a sub-optimal upper bound over $\mathcal{C}_{\alpha}$ of $\|\hat \Sigma_K - \Sigma\|_{op}$, and our contribution. We begin by taking a closer look at the arguments used in \cite{Bickel:08b}. We shall use the fact that with high probability, $\max_{i,j} |\hat\sigma_{ij}-\sigma_{ij}| \lesssim \sqrt{\log p/n}$; see equality (12) of  \cite{Bickel:08a}.
The following inequality will also be used multiple times: for any symmetric real matrix $A$,
\begin{equation}
\label{11norm}
\|A\|_{op}\leq \|A\|_{1,1},
\end{equation}
where $\|A\|_{1,1} = \max_{i} \sum_j |a_{ij}|$.

The derivation in \cite{Bickel:08b} is essentially as follows. By the triangle inequality,
\begin{align*}
\|\hat \Sigma_K -\Sigma\|_{op} \leq \|\hat \Sigma_K- \Sigma_K\|_{op} + \|\Sigma-\Sigma_K\|_{op},
\end{align*}
where $\Sigma_K = (\sigma_{ij}1_{|i-j|\leq K-1})_{1\leq i, j\leq p}$.
For the term $ \|\Sigma-\Sigma_K\|_{op}$ with $\Sigma\in\mathcal{C}_{\alpha}$,  we have

\begin{align}
\|\Sigma-\Sigma_K\|_{op} \leq&\|\Sigma-\Sigma_K\|_{1,1}\label{11norm1}\\
 =&\max_i \sum_{|i-j|\geq K} |\sigma_{ij}|\nonumber\\
 \lesssim &K^{-\alpha}.\nonumber
\end{align}
Then,
\begin{align}
\| \hat \Sigma_K- \Sigma_K\|_{op}\leq& \|\hat \Sigma_K- \Sigma_K\|_{1,1}\label{11norm2}\\
=&\max_i \sum_{|i-j|\leq K-1} |\hat\sigma_{ij}-\sigma_{ij}|\nonumber\\
\leq&(2K-1)\max_{ij}  |\hat\sigma_{ij}-\sigma_{ij}|\nonumber\\
\lesssim& K\sqrt{\log p/n}\nonumber
\end{align}
with high probability. Therefore,
$$
\|\hat \Sigma_K -\Sigma\|_{op} \lesssim K\sqrt{\log p/n} + K^{-\alpha}, \text{ for any } K,
$$
with high probability.
By choosing $K\asymp (\log p/n)^{-\frac{1}{2(\alpha+1)}}$, one obtains that
$$
\|\hat \Sigma_K -\Sigma\|_{op} \lesssim \left(\frac{\log p}{n}\right)^{\frac{\alpha}{2(\alpha+1)}}
$$
with high probability, which is sub-optimal.

It is easy to see that the inequality~(\ref{11norm1}) is tight over $\mathcal{C}_{\alpha}$  and cannot be improved. However,
the inequality~(\ref{11norm2}) is not tight. We show in Proposition~\ref{prop_band} that  inequality~(\ref{11norm}) can be
 reduced by an important $\sqrt{K}$ factor. With this improvement and by choosing  $K \asymp n^{-1/(2\alpha+1)}$, we can show that 
 indeed, 
$$
\|\hat \Sigma_K -\Sigma\|_{op} \lesssim n^{-\frac{\alpha}{2\alpha+1}} + \sqrt{\log p/n}
$$
with high probability, which is the optimal rate given in Theorem~\ref{thm}. Moreover, the bound in expectation of Theorem~\ref{thm} can be similarly derived from Proposition~\ref{prop_band}.

\begin{prop}
\label{prop_band}
For the banding estimator $\hat\Sigma_K$ with $\Sigma\in \mathcal{C}_{\alpha}$, 
there exists a constant $c>0$ such that
$$
\mathbb{P}\left\{ \|\hat\Sigma_K - \Sigma_K\|_{op} \geq   c \|\Sigma^{abs}\|_{op}\sqrt{\frac{K + \log p}{n}}\right\}\lesssim p^{-1}.
$$
Furthermore,
$$
\mathbb{E} \|\hat \Sigma_K - \Sigma_K\|_{op}^2 \lesssim \frac{K +\log p}{n}.
$$
\end{prop}
\begin{proof}
For simplicity we assume $p/K$ is an integer, the case when $p/K$ is not an integer  can be similarly handled with slightly more technical complexity.   For a
$p\times p$ matrix $A$, let $A(k,\ell)$  denote
the $(k,\ell)$ submatrix (or ``block'') of the form 
\begin{equation}
\label{block}
\{A_{ij}:(i,j)\in[(k-1)K + 1,kK]\times[(\ell-1)K + 1,\ell K]\}.
\end{equation}
Also, let $A^{\ast}$ denote the $p/K\times p/K$ matrix with $(k,\ell)$ entry $\|A(k, \ell)\|_{op}$. Note that $A^{\ast}$ will be symmetric if $A$ is so.

We divide $\hat\Sigma_K$ into $(p/K)^2$ blocks of dimension $K\times
K$. First note that the number of non-zero blocks in each row or column of the blocks of  $\hat \Sigma_K$ is at most 3. To see this, consider the $(k,\ell)$th block $A(k,\ell)$ and assume it contains the $(i,j)$th element of $\hat \Sigma_K$. Then by the definition in~(\ref{block}), $ (k-1)*K + 1 \leq i \leq kK$ and $(\ell-1)*K + 1\leq j\leq \ell K$. If $k \geq \ell + 2$, then $i-j \geq (k-1)*K +1 - \ell K \geq K + 1$ and hence $\hat \sigma $ is zero. Hence if $k \geq \ell + 2$, and similarly if $\ell \geq k + 2$, $A(k,\ell)$ contains only zero elements. In other words, $A(k,\ell)$ might be non-zero only if $|k-\ell| \leq 1$. There are two types of blocks in $\hat \Sigma_K$ with
$|k-\ell|\leq 1$: diagonal blocks with $k=\ell$ and non-diagonal
blocks with $k\neq \ell$. For the diagonal blocks,  $\hat
\Sigma_K(k,k) = \hat \Sigma(k,k)$ and $\Sigma_K(k,k) =
\Sigma(k,k)$. Let $H_0 = (1_{\{i>j\}})_{1\leq i, j\leq K}$ be a strictly
lower-triangular matrix of ones. The off-diagonal matrices $\hat \Sigma_K(k,\ell)$ with $k-\ell = \pm 1$ have two forms: $\hat \Sigma(k,\ell)\ast H_0$ if $k<\ell$ and $\hat\Sigma(k,\ell) \ast H_0^T$  if $k >\ell$. Here $\ast$ is the Schur matrix multiplication. Similarly $\Sigma_K(k,\ell)$ is  $\Sigma(k,\ell)\ast H_0$ if $k<\ell$ and is $\Sigma(k,\ell)\ast H_0^T$ if $k>\ell$. See Figure~\ref{fig_blocks} for an illustration. All three forms of blocks in $\hat \Sigma_K$ have the general form $\hat\Sigma(k,\ell)*H$ for a $K\times K$ matrix $H = (h_{k\ell})$ with $|h_{k\ell}|\leq 1$ for all $(k,\ell)$.

\begin{figure}[htp]
\centering
\includegraphics[width=5in,  angle=0]{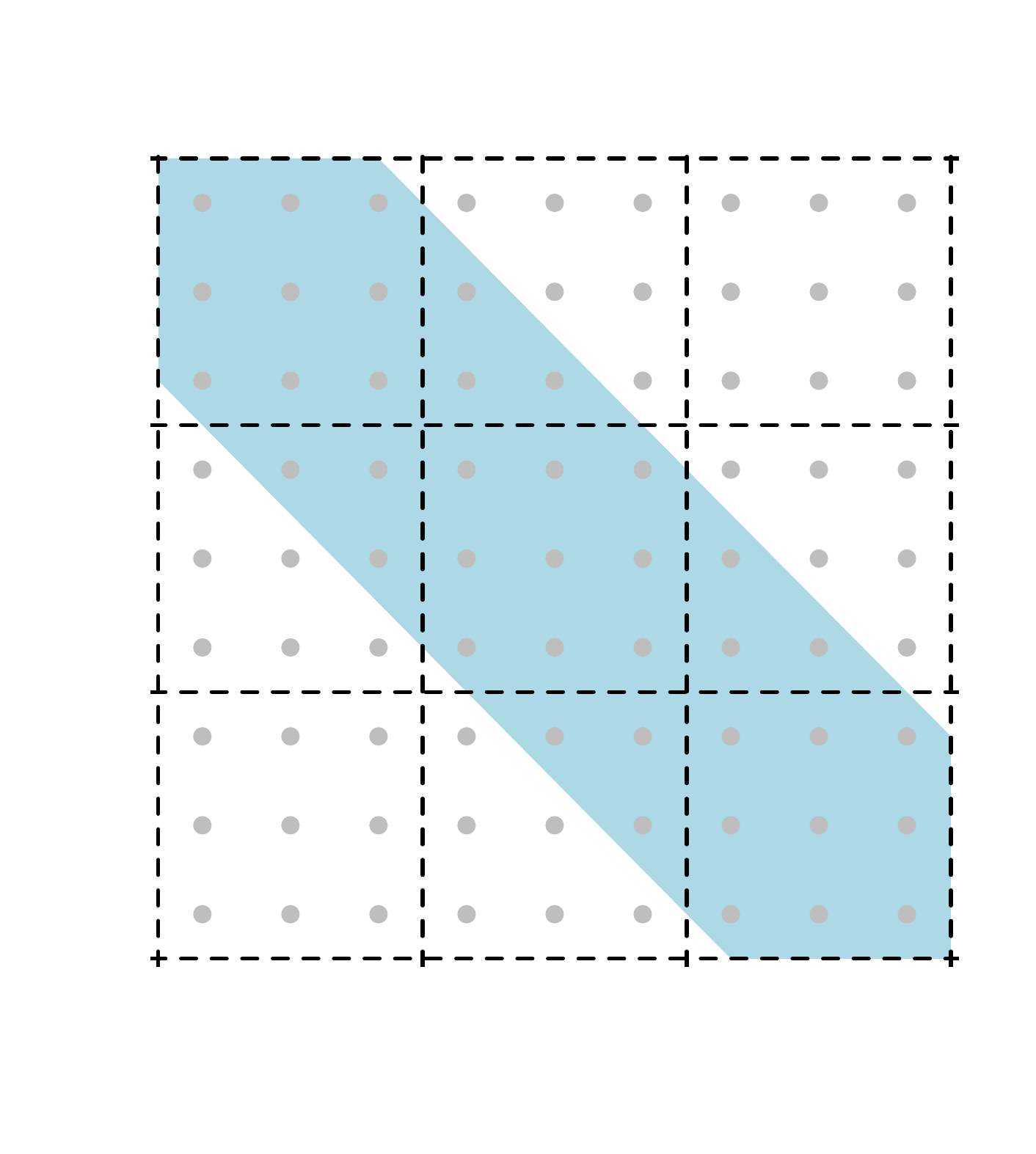}
\caption{\label{fig_blocks} An illustration of block partition for the banding estimator with $p=9$ and $K=3$. Each filled circle is an entry and these entries outside the shaded area are truncated to 0. The off-diagonal blocks have only $K*(K-1)/2 = 3$ non-zero entries.}
\end{figure}

Now we know  each row or column of $(\hat \Sigma_K-\Sigma_K)^{\ast}$ has at most three non-zero entries.
By the norm compression inequality in \cite{Cai:12b},
$$
\|\hat \Sigma_K-\Sigma_K\|_{op}\leq \|(\hat \Sigma_K -\Sigma_K)^{\ast}\|_{op}\leq 3\max_{|k-\ell|\leq 1} \|\hat \Sigma_K(k,\ell) - \Sigma_K(k,\ell)\|_{op},
$$
where the second inequality follows by~(\ref{11norm}).
Then Proposition~\ref{prop_band} is proved by
Proposition~\ref{prop_max}. 
\end{proof}

\begin{prop}
\label{prop_max}
There exists a constant $c>0$ such that
$$
\mathbb{P}\left\{\max_{|k-\ell| \leq 1} \|\hat\Sigma_K(k,\ell) - \Sigma_K(k,\ell)\|_{op} \geq  c \|\Sigma^{abs}\|_{op}\sqrt{\frac{K + \log p}{n}}\right\}\lesssim p^{-1}.
$$
Furthermore,
$$
\mathbb{E}\max_{|k-\ell|\leq 1} \|\hat \Sigma_K(k,\ell) - \Sigma_K(k,\ell)\|_{op}^2 \lesssim \frac{K +\log p}{n}.
$$
\end{prop}
\begin{rem}
The proposition  provides probability and risk bounds for  sample covariance matrix and for  sample cross covariance matrix with upper triangular elements fixed at zero, and hence 
extends results in \cite{Bunea:13}, which  considers only   sample covariance matrix,  and also
complements Lemma 2 in \cite{Cai:12b}, which considers sample cross covariance matrices.
\end{rem}
\begin{rem}
The probability bound on $\|\hat\Sigma_K(k,\ell) -
\Sigma_K(k,\ell)\|_{op}$ for $k-\ell =\pm 1$ is non-standard,  as the
matrices involved have irregular forms and  fixed zero entries. The
derivation of this bound  requires concentration inequalities for the
bilinear form $X^TAY$ where $X$ and $Y$ are multivariate random
vectors and $A$ is an arbitrary non-random matrix. To our best
knowledge, such concentration inequalities have not been derived in
the literature,  in which only the special case $X = Y$ and $A$ is symmetric have been treated, see e.g. \cite{Boucheron:13}. We provide these inequalities in Proposition~\ref{prop2} in the appendix.
\end{rem}

The proof of Proposition~\ref{prop_max} is deferred to Section~\ref{sec:derivation}.

\section{Sure-tuned Bandwidth Selection}\label{sec:sure}
This section is devoted to the selection of the bandwidth of an
operator-norm  accurate estimator.  One possibility, as in
\cite{Bickel:08b}, is to use cross validation. If operator norm is used for defining the loss function, cross validation can be computationally quite intensive for large $p$ because about $O(p)$ operator norms of  $p\times p$ 
matrices have to be evaluated. \cite{Bickel:08b} used the maximum row sum norm $\|\cdot\|_{1,1}$ for defining the loss function. We propose an alternative approach, with low computational complexity.  Our procedure minimizes in $K$ a data-driven criterion  that is a function of  the bandwidth $K$. The proposed  criterion is a  {\it modified} unbiased estimator of the Frobenius-norm risk of $\hat{\Sigma}_K$, and its derivation follows the general principles  of Stein's unbiased risk estimation (Sure) \citep{Stein:81}.  

 To begin, note that the Frobenius risk  of $\hat \Sigma_K$ is 
\begin{align*}
\mathbb{E}\|\hat \Sigma_K-\Sigma\|_F^2&=\sum_{|i-j|\leq K-1} \mathbb{E}(\hat\sigma_{ij} - \sigma_{ij})^2 + \sum_{|i-j|\geq K} \sigma_{ij}^2\\
&=\sum_{|i-j|\leq K-1} \text{Var}(\hat\sigma_{ij}) + \sum_{|i-j|\geq K} \sigma_{ij}^2.
\end{align*}
The first term above  is the sum of variances of the entries in $\hat \Sigma_K$ while the second term is the sum of
squared biases of $\hat \Sigma_K$. The following proposition provides unbiased estimates of $\text{Var}(\hat\sigma_{ij})$ and $\sigma_{ij}^2$.
\begin{prop}
\label{prop1}
\begin{equation}
\label{prop1_eq1}
\text{Var}(\hat\sigma_{ij}) = \frac{\sigma_{ii}\sigma_{jj} + \sigma_{ij}^2}{n-1}
\end{equation}
and an unbiased estimate of $\text{Var}(\hat\sigma_{ij})$ can be given by 
\begin{equation*}
\widehat{\text{Var}}(\hat\sigma_{ij}) = a_n \hat\sigma_{ii}\hat\sigma_{jj} + b_n\hat\sigma_{ij}^2,
\end{equation*}
where $a_n = \frac{n-1}{n^2-n-2}$ and $b_n = \frac{n-3}{n^2-n-2}$.
Moreover, an unbiased estimate of $\sigma_{ij}^2$ can be given by 
$
c_n \hat\sigma_{ii}\hat\sigma_{jj} + d_n\hat\sigma_{ij}^2,
$
where $c_n = \frac{1-n}{n^2-n-2}$ and $d_n = \frac{(n-1)^2}{n^2-n-2}$.
\end{prop}
\begin{rem}
The proof is omitted as it follows straightforwardly from Lemma~\ref{lem1} in the appendix.  Equation~(\ref{prop1_eq1}) was first derived by \cite{Yi:13}. \end{rem}

\noindent By Proposition~\ref{prop1}, an unbiased estimate of the Frobenius  risk of $\hat\Sigma_K$ can therefore be given by: 

\begin{equation}\label{Fnorm}
\mbox{Sure}_F(K) = \sum_{|i-j|\leq K-1} \left( a_n \hat\sigma_{ii}\hat\sigma_{jj} + b_n \hat\sigma_{ij}^2\right)+\sum_{|i-j|\geq K}\left(c_n \hat\sigma_{ii}\hat\sigma_{jj} + d_n\hat\sigma_{ij}^2\right).
\end{equation}
One could then select
$$
\hat{K}_F = \arg\min_{K} \text{Sure}_F(K).
$$
A similar procedure has been suggested in \cite{Yi:13},  but for tapering estimators. 

We denote by $\hat P_{F}$ the Sure-tuned banding estimator with bandwidth equal to $\hat K_F$.
The estimator $\hat P_F$ is appropriate if the goal is to construct a Frobenius-norm accurate estimator.  However, it is known from the theoretic analysis in \citep{Cai:10}
that the  bandwidth for optimal  Frobenius norm estimation is asymptotically smaller  than what is needed for optimal operator norm estimation.  We  propose some
modification to criterion~(\ref{Fnorm}) above, that will encourage the
selection of a larger bandwidth. The  idea   is to  place a larger  weight  on the bias term, which is the second sum in~(\ref{Fnorm}), so that a larger bandwidth is selected. 
 We do this via the factor $K$ in the weights $W_{ijK}$ given by~(\ref{W}) below. Moreover, we notice that, over the class $\mathcal{C}_{\alpha}$, the entries $\sigma_{ij}^2$ corresponding to large $|i-j|$ are small,  but their estimates $c_n \hat\sigma_{ii}\hat\sigma_{jj} + d_n\hat\sigma_{ij}^2$, albeit unbiased, have variability  that can be much higher than the size of  $\sigma_{ij}^2$. Therefore, for a more stable selection of $K$, we attenuate the contribution of the estimates of $\sigma_{ij}^2$ with large $|i-j|$ via the exponentially decaying factor in~(\ref{W}).

Therefore we use the following criterion
$$
\text{Sure}_{op}(K) = \sum_{|i-j|\leq K-1} \left( a_n \hat\sigma_{ii}\hat\sigma_{jj} + b_n \hat\sigma_{ij}^2\right)+\sum_{|i-j|\geq K} W_{ijK}\left(c_n \hat\sigma_{ii}\hat\sigma_{jj} + d_n\hat\sigma_{ij}^2\right),
$$
where 
\begin{equation}
\label{W}
W_{ijK} = K\exp\left(1-\frac{|i-j|}{K}\right),
\end{equation}
and select 
$$
\hat{K}_{op} = \arg\min_{K} \text{Sure}_{op}(K).
$$

We call the banding estimator with bandwidth $\hat{K}_{op}$ the
``modified Sure-tuned banding estimator" and  denote it by $\hat P_{op}$.
 We now give a heuristic  argument why the above approach might select
 a bandwidth that is well-suited for estimation under the operator norm. We assume $|\sigma_{ij}|\leq M_1 |i-j|^{-\alpha - 1}$ for all $(i,j)$. Then
\begin{align*}
\mathbb{E} \text{Sure}_{op}(K) &= \sum_{|i-j|\leq K-1} \text{Var}(\hat\sigma_{ij}) + \sum_{|i-j|\geq K} KW_{ijK}\sigma_{ij}^2\\
& = O(K/n) + O\left\{\sum_{|i-j|\geq K} \exp\left(1-\frac{|i-j|}{K}\right)K^{-(2\alpha+1)}\right\}\\
& = O(K/n) + O(K^{-2\alpha}).
\end{align*}
Hence $\mathbb{E} \text{Sure}_{op}(K)$ is minimized only if $K = O\left(n^{-\frac{1}{2\alpha+1}}\right)$, and we recall that in Remark 1 above we showed that the optimal bandwidth for operator-norm estimation is of this order.  We further demonstrate experimentally in the following section that the estimator with a bandwidth thus selected has excellent operator norm behavior.



\section{Simulations}\label{sec:sim}
We compare the following 6 estimators:

\begin{enumerate}[(i)]
 \item $\hat P_{cv}$: the banding estimator for which the bandwidth is  selected by 10-fold cross validation with squared operator norm as the loss function ;
 \item $\hat P_{BL}$: the banding estimator in \cite{Bickel:08b} for which the bandwidth is selected by 10-fold cross validation with the maximum row sum norm $\|\cdot\|_{1,1}$ as the loss function;
 \item $\hat P_{CY}$: the block-thresholding estimator in \cite{Cai:12b}; 
  \item $\hat P_{YZ}$:  the Sure-tuned tapering estimator in
    \cite{Yi:13}; 
 \item $\hat P_{F}$: the Sure-tuned banding estimator;
 \item  $\hat P_{op}$: the modified Sure-tuned banding estimator.
\end{enumerate}

The data are generated from $\mathcal{N}(0,\Sigma$). 
Following \cite{Bickel:08b} and \cite{Cai:10}, the covariance matrix $\Sigma$ has the following form
$$
\sigma_{ij} = \rho |i-j|^{-(\alpha+1)},
$$
where $\rho = 0.6$ and $\alpha$ can be either $0.1$ or $0.5$. Similar to \cite{Yi:13}, we fix $n$ at 250 and let $p$ be either  of 250, 500 and 1000. For each scenario, we run 100 simulations and compute the mean squared errors in terms of the operator norm. For example, for the re-weighted Sure-tuned banding estimator $\hat P_{op}$, its mean squared error is 
$$
\frac{1}{100}\sum_{k=1}^{100}\|\hat P_{op}^k -\Sigma\|_{op}^2,
$$
where $\hat P_{op}^k$ is the estimate for the $k$th simulated
dataset. It is noted by \cite{Yi:13} that the optimal bandwidth for
the operator norm can be quite variable. To reduce the variability of the selected bandwidth,  $\hat K_{op}$ will be restricted to the interval $[\hat K_{F},\hat K_F^2]$.

Table~\ref{table1} gives the simulation results.  Several observations can be made from Table~\ref{table1}. First, the estimator $\hat P_{cv}$ using cross-validation, one of the most widely used statistical techniques,  has  the worst performance in this problem. We note that calculation of $\hat P_{cv}$ is also very time consuming when $p$ is large. Secondly,  it is interesting to see that the operator-norm rate-optimal $\hat P_{CY}$ is  dominated by three other Sure-type estimators, $\hat P_{FZ}$, $\hat P_{F}$ and $\hat P_{op}$. Third, the Sure-tuned banding and tapering estimators, $\hat P_{FZ}$ and $\hat P_{F}$, have  comparable MSEs  and the modified Sure-tuned banding estimator $\hat P_{op}$ always has the smallest MSEs, except for one scenario. The modified Sure-tuned banding estimator $\hat P_{op}$  has larger standard error than $\hat P_{FZ}$ and $\hat P_{F}$  because of the larger variability of the selected bandwidth (results not shown).

\begin{table}
\centering
\caption{\label{table1} Mean and standard deviations (in parenthesis) of squared errors in operator norm for the 6 estimators.}
\begin{tabular}{|cccccccc|}
\hline
$\alpha$&$p$
&$\hat P_{cv}$
&$\hat P_{BL}$
&$\hat P_{CY}$
&$\hat P_{YZ}$
&$\hat P_{F}$
&$\hat P_{op}$\\
\hline
\multirow{3}{*}{$0.1$}&250&7.96 (4.59)&6.07 (3.27)&13.34 (0.26)&5.36 (0.67)&5.38 (0.64)&4.61 (1.37)\\
&500&17.68 (12.32)& 8.36 (5.01) &15.86 (0.24)&7.73 (0.69)&7.85 (0.65)&6.05 (1.51)\\
&1,000&37.20 (28.78)&10.88 (7.21)&19.81 (0.18)&10.59 (0.60)&10.56 (0.49)&8.16 (1.78)\\
\hline
\multirow{3}{*}{$0.5$}&250&6.08 (4.60)&3.48 (3.23)&2.72 (0.12)&1.08 (0.13)&1.07 (0.14)&1.13 (0.32)\\
&500&12.59 (11.01) &4.44 (5.88)&2.62 (0.09)&1.22 (0.08)&1.21 (0.10)&1.18 (0.23)\\
&1,000&31.88 (32.18)&6.51(14.10)&2.79 (0.07)&1.35 (0.07)&1.33 (0.07)&1.27 (0.32)\\
\hline
\end{tabular}
\end{table}

%
\section{Proof of Proposition~\ref{prop_max}}\label{sec:derivation}

\begin{proof}

We start with studying $u^T\left\{\hat \Sigma(k,\ell)*H - \Sigma(k,\ell)*H\right\}v$, where $u, v\in\mathbb{S}^{K-1}$. Assume $X$ and $Y$ have a  joint real normal distribution 
$\textnormal{MVN}_{2K}(0, \Sigma^1)$ with 
$$
\Sigma^1 = \left(\begin{array}{cc}\Sigma_{11}&\Sigma_{12}\\
\Sigma_{21}&\Sigma_{22}\end{array}\right) 
$$
and 
$$
\Sigma_{11} = \Sigma(k,k), \Sigma_{22} = \Sigma(\ell,\ell), \Sigma_{12} =\Sigma_{21}^T = \Sigma(k,\ell).
$$
Let $A = (uv^T)*H$.
Then $u^T\left\{\hat \Sigma(k,\ell)*H - \Sigma(k,\ell)*H\right\}v$ is the same in distribution as
$$
 \frac{1}{n-1}\sum_{i=1}^n (X_i-\bar X)^TA(Y_i-\bar Y) -\tr(A\Sigma_{21})  ,
$$
where $(X_1,Y_1),\dots, (X_n,Y_n)$ are i.i.d.\ copies of $(X,Y)$, $\bar X = n^{-1}\sum_{i=1}^n X_i$, and
 $\bar Y = n^{-1}\sum_{i=1}^n Y_i$. It is easy to show that $\mathbb{E}(X_i^TAY_i) = \tr(A\Sigma_{21})$ and $ \mathbb{E}(\bar X^T A\bar Y) = n^{-1}\tr(A\Sigma_{21})$. Therefore,
\begin{equation}
\label{Q0}
\begin{split}
 &\frac{1}{n-1}\sum_{i=1}^n (X_i-\bar X)^TA(Y_i-\bar Y) -\tr(A\Sigma_{21}) \\
  =&\frac{n}{n-1}\left[\frac{1}{n}\sum_{i=1}^n \left\{X_i^TAY_i- \mathbb{E}(X_i^TAY_i)\right\} - \left\{\bar X^T A\bar Y - \mathbb{E}(\bar X^T A\bar Y)\right\}\right].
  \end{split}
\end{equation}

We first consider the term $n^{-1}\sum_{i=1}^n \left\{X_i^TAY_i- \mathbb{E}(X_i^TAY_i)\right\}$  in~(\ref{Q0}). Let $Q_i = X_i^TA Y_i, i=1,\dots, n$, then $Q_1,\dots, Q_n$ are i.i.d.\ copies of $Q = X^T AY$. Let $\bar Q = n^{-1}\sum_{i=1}^n (Q_i - \mathbb{E}Q_i)$, which equals $n^{-1}\sum_{i=1}^n \left\{X_i^TAY_i- \mathbb{E}(X_i^TAY_i)\right\}$ in distribution.  By Proposition~\ref{prop2}, 
for $0 < t< 1/2$,
$$
\mathbb{P}\left\{|\bar Q|>t \sqrt{\tr(B^2)}\right\} \leq 2\exp\left(- \frac{nt^2}{2}\right),
$$
where
$$
B = \Sigma^{1,1/2}\left(\begin{array}{cc} 0_{K, K}& A\\
A^T&0_{K,K}\end{array}\right)\Sigma^{1,1/2}.
$$
By Lemma~\ref{lem6} we have
$$
\tr(B^2) \leq 4\|\Sigma^{1,abs}\|_{op}^2\leq 4\|\Sigma^{abs}\|_{op}^2
$$
Therefore,
$$
\mathbb{P}\left\{|\bar Q|>2t \|\Sigma^{abs}\|_{op}\right\} \leq 2\exp\left(- \frac{nt^2}{2}\right),
$$
or equivalently,
\begin{equation}
\label{Q1}
\mathbb{P}\left\{\left|\frac{1}{n}\sum_{i=1}^n \left\{X_i^TAY_i- \mathbb{E}(X_i^TAY_i)\right\}\right|>2t \|\Sigma^{abs}\|_{op}\right\} \leq 2\exp\left(- \frac{nt^2}{2}\right).
\end{equation}

We next consider the term $\bar X^T A\bar Y - \mathbb{E}(\bar X^T A\bar Y)$ in~(\ref{Q0}). Note that $\bar X$ and $\bar Y$ have a joint real normal distribution $\text{MVN}(0, n^{-1}\Sigma^1)$. By similar derivation as above,
$$
\mathbb{P}\left\{ \left|\bar X^T A\bar Y - \mathbb{E}(\bar X^T A\bar Y) \right|> 2 t \|\Sigma^{abs}\|_{op}/n\right\}\leq 2\exp\left(- \frac{t^2}{2}\right),
$$
or equivalently,
\begin{equation}
\label{Q2}
\mathbb{P}\left\{ \left|\bar X^T A\bar Y - \mathbb{E}(\bar X^T A\bar Y) \right|> 2 t \|\Sigma^{abs}\|_{op}\right\}\leq 2\exp\left(- \frac{n^2t^2}{2}\right).
\end{equation}

Note that for any two random variables $Z_1$ and $Z_2$, 
$$
\mathbb{P}(|Z_1 - Z_1|>2x) \leq \mathbb{P}(|Z_1|>x) + \mathbb{P}(|Z_2|>x).
$$
Combining~(\ref{Q0}), (\ref{Q1}) and~(\ref{Q2}), we obtain
\begin{align*}
&\mathbb{P}\left[\left|u^T\left\{\hat \Sigma(k,\ell)*H - \Sigma(k,\ell)*H\right\}v\right|>\frac{4(n-1)}{n}\|\Sigma^{abs}\|_{op} t\right]\\
\leq& 2\exp\left(- \frac{nt^2}{2}\right) + 2\exp\left(- \frac{n^2t^2}{2}\right),
\end{align*}
which leads to
\begin{equation}
\label{der_eq1}
\mathbb{P}\left[\left|u^T\left\{\hat \Sigma(k,\ell)*H - \Sigma(k,\ell)*H\right\}v\right|>4\|\Sigma^{abs}\|_{op} t\right]\leq 4\exp\left(- \frac{nt^2}{2}\right).\end{equation}

Now we consider $\|\hat \Sigma(k,\ell)*H - \Sigma(k,\ell)*H\|_{op}$.
By Lemma~\ref{lem8}, there exists an $\delta$-net $\mathbb{Q}_K\in\mathbb{S}^{K-1}$ such that 
$$
\text{card}(\mathbb{Q}_K)\leq c_1\delta^{-K}K^{3/2}\log (1+K)
$$
for some constant $c_1>0$. It can also be shown that
\begin{align*}
&\|\hat \Sigma(k,\ell)*H - \Sigma(k,\ell)*H\|_{op}\\
=&\sup_{u, v\in\mathbb{S}^{K-1}}
 u^T\left\{\hat \Sigma(k,\ell)*H - \Sigma(k,\ell)*H\right\}v.
\\
 \geq& (1-2\delta)^{-1}\sup_{u, v\in\mathbb{Q}^{K-1}}
 u^T\left\{\hat \Sigma(k,\ell)*H - \Sigma(k,\ell)*H\right\}v.
\end{align*}
By~(\ref{der_eq1}) and the union bound,
\begin{align*}
&\mathbb{P}\left[\max_{|k-\ell|\leq 1}\|\hat \Sigma(k,\ell)*H - \Sigma(k,\ell)*H\|_{op} >t\right]\\
\leq&
\mathbb{P}\left[\max_{|k-\ell|\leq 1}\sup_{u, v\in \mathbb{Q}_K}u^T\left\{\hat \Sigma(k,\ell)*H - \Sigma(k,\ell)*H\right\}v>(1-2\delta)t \right]\\
\leq& 12c_1p\delta^{-2K}K^{3}\log^2 (1+K)\exp\left\{- \frac{nt^2(1-2\delta)^2}{32\|\Sigma^{abs}\|_{op}^2}\right\}.
\end{align*}
To summarize, if we let $W = \max_{|k-\ell|\leq 1}\|\hat \Sigma(k,\ell)*H - \Sigma(k,\ell)*H\|_{op}$, then
\begin{equation}
\label{der_eq2}
\mathbb{P}(W>t) \leq 12cp\delta^{-2K}K^{3}\log^2 (1+K)\exp\left\{- \frac{nt^2(1-2\delta)^2}{32\|\Sigma^{abs}\|_{op}^2}\right\}.
\end{equation}

To establish the probability bound in Proposition~\ref{prop_max}, 
similar to \cite{Cai:12b} we just need to rewrite~(\ref{der_eq2}) by letting $\delta = \exp(-3)$, $c_0 = \sqrt{192}/(1-2\delta)$ and
$$
t = c_0 \|\Sigma^{abs}\|_{op}\sqrt{\frac{K + \log p}{n}}.
$$ 
The inequality in expectation can be similarly derived by~(\ref{der_eq2}) and the fact that
$$
\mathbb{E}W^2 \leq  x + \int_x^{\infty} \mathbb{P}(W^2>t)dt
$$
for any $x>0$.
\end{proof}

\appendix

\section{A Lemma for Proposition~\ref{prop1}}
\begin{lem}
\label{lem1}
\begin{eqnarray}
\mathbb{E}(\hat\sigma_{ij}^2) &=& \frac{1}{n-1}\sigma_{ii}\sigma_{jj} + \frac{n}{n-1}\sigma_{ij}^2,\label{lem1_eq1}\\
\mathbb{E}(\hat\sigma_{ii}\hat\sigma_{jj}) &=&\sigma_{ii}\sigma_{jj} + \frac{2}{n-1}\sigma_{ij}^2.\label{lem1_eq2}
\end{eqnarray}
\end{lem}
\begin{rem}
\cite{Yi:13} derived the above equalities, however, their result on $\mathbb{E}(\hat\sigma_{ii}\hat\sigma_{jj})$ is incorrect; see equations (A.7) and (A.12) therein. A quick check is to let $i=j$.
\end{rem}
\begin{proof}
We assume w.l.o.g.\ that $\mu=0$. We use $\bar x_i$ to denote  $n^{-1}\sum_{k=1}^n x_{k,i}$.
Note that he following equation for all pairs of $(i,j)$,
\begin{equation}
\label{lem1_eq3}
\mathbb{E}(x_i^2x_j^2) = \sigma_{ii}\sigma_{jj} + 2\sigma_{ij}^2.
\end{equation}
It's straightforward to show that 
$$
(n-1)\hat\sigma_{ij} = \sum_{k=1}^n x_{k,i}x_{k,j}  - n\bar x_i \bar x_j.
$$
Note that 
$$
(\bar x_i,\bar x_j ,x_{1,i}, x_{1,j})^T \sim \text{MVN}\left(0, 
\left(
\begin{array}{cccc}
\frac{\sigma_{ii}}{n}&\frac{\sigma_{ij}}{n}&\frac{\sigma_{ii}}{n}&\frac{\sigma_{ij}}{n}\\
\frac{\sigma_{ij}}{n}&\frac{\sigma_{jj}}{n}&\frac{\sigma_{ij}}{n}&\frac{\sigma_{ij}}{n}\\
\frac{\sigma_{ii}}{n}&\frac{\sigma_{ij}}{n}&\sigma_{ii}&\sigma_{ij}\\
\frac{\sigma_{ij}}{n}&\frac{\sigma_{ij}}{n}&\sigma_{ij}&\sigma_{jj}
\end{array}
\right)
\right).
$$
Hence 
\begin{eqnarray}
\label{lem1_eq4}
\mathbb{E}(\bar x_i^2 x_{1,j}^2) &=& \sigma_{ii}\sigma_{jj}/n + 2\sigma_{ij}^2/n,\\
\label{lem1_eq5}
\mathbb{E}(\bar x_i^2 \bar x_{j}^2) &=& \sigma_{ii}\sigma_{jj}/n^2 + 2\sigma_{ij}^2/n^2,\\
\label{lem1_eq6}
\mathbb{E} (\bar x_i \bar x_j x_{1,i} x_{1,j}) &=&\sigma_{ii}\sigma_{jj}/n^2+ \sigma_{ij}^2(1/n + 1/n^2) .
\end{eqnarray}

We first derive~(\ref{lem1_eq1}) by using~(\ref{lem1_eq3}), (\ref{lem1_eq5}) and~(\ref{lem1_eq6}). We have
\begin{align*}
&(n-1)^2\mathbb{E} (\hat\sigma_{ij}^2)\\
 &= \mathbb{E}\left (\sum_{k=1}^n x_{k,i}x_{k,j}  - n\bar x_i \bar x_j\right )^2\\
&=\mathbb{E}\left (\sum_{k=1}^n x_{k,i}x_{k,j}\right )^2 - 2n\mathbb{E}\left(\bar x_i \bar x_j \sum_{k=1}^n x_{k,i}x_{k,j} \right) + n^2 \mathbb{E}(\bar x_i^2 \bar x_j^2)\\
& = \sum_{k, k^{\prime}}\mathbb{E}(x_{k,i}x_{k,j} x_{k^{\prime}, i} x_{k^{\prime}, j}) - 2n^2 \mathbb{E} (\bar x_i \bar x_j x_{1,i} x_{1,j}) + n^2 \mathbb{E}(\bar x_i^2 \bar x_j^2)\\
& = \sum_{k=k^{\prime}} (\sigma_{ii}\sigma_{jj} + 2\sigma_{ij}^2) + \sum_{k\neq k^{\prime}}\sigma_{ij}^2 - 2(\sigma_{ii}\sigma_{jj} +  (n+1)\sigma_{ij}^2) + (\sigma_{ii}\sigma_{jj} + 2\sigma_{ij}^2)\\
& = (n-1)\sigma_{ii}\sigma_{jj} +n(n-1)\sigma_{ij}^2,
\end{align*}
which proves~(\ref{lem1_eq1}).

Next we derive~(\ref{lem1_eq2}) by using~(\ref{lem1_eq3}), (\ref{lem1_eq4}) and~(\ref{lem1_eq5}). We have
\begin{align*}
&(n-1)^2\mathbb{E}(\hat\sigma_{ii}\hat\sigma_{jj})\\
&=\mathbb{E}\left\{\left(\sum_{k=1}^n x_{k,i}^2  - n\bar x_i^2\right)\left(\sum_{k=1}^n x_{k,j}^2  - n \bar x_j^2\right)\right\}\\
&=\sum_{k,k^{\prime}} \mathbb{E}(x_{k,i}^2 x_{k^{\prime},j}^2) - n\sum_k \mathbb{E}(\bar{x}_i^2 x_{k,j}^2)-n\sum_k \mathbb{E}(\bar{x}_j^2 x_{k,i}^2) + n^2\mathbb{E}(\bar x_i^2 \bar x_j^2)\\
&=\sum_{k=k^{\prime}} (\sigma_{ii}\sigma_{jj} + 2\sigma_{ij}^2) + \sum_{k\neq k^{\prime}}(\sigma_{ii}\sigma_{jj}) - n^2\mathbb{E}(\bar x_i^2 x_{1,j}^2) - n^2\mathbb{E}(\bar x_j^2 x_{1,i}^2) + n^2\mathbb{E}(\bar x_i^2\bar x_j^2)\\
&=n^2\sigma_{ii}\sigma_{jj} + 2n\sigma_{ij}^2 - 2(n\sigma_{ii}\sigma_{jj} + 2\sigma_{ij}^2) + (\sigma_{ii}\sigma_{jj} + 2\sigma_{ij}^2)\\
&= (n-1)^2\sigma_{ii}\sigma_{jj} + 2(n-1)\sigma_{ij}^2,
\end{align*}
which proves~(\ref{lem1_eq2}).

\end{proof}

\section{Supplemental materials for Section~\ref{sec:derivation}}
\begin{prop}
\label{prop2}
Let the $p\times 1$ vector $X$ and $q\times 1$ vector $Y$ have a joint real normal distribution $\text{MVN}_{p+q}(0,\Sigma), \Sigma>0$. Let $A$ be a $p\times q $ real matrix. Let $Q = X^T AY$. Let $Q_1, \dots, Q_n$ be i.i.d.\ realizations of $Q$. Denote $\bar Q = n^{-1}\sum_{i=1}^n(Q_i-\mathbb{E}Q_i)$. Then, for $0 < t< 1/2$,
$$
\mathbb{P}\left\{|\bar Q|>t \sqrt{\tr(B^2)}\right\} \leq 2\exp\left(- \frac{nt^2}{2}\right),
$$
where
$$
B = \Sigma^{1/2}\left(\begin{array}{cc} 0_{K, K}& A\\
A^T&0_{K,K}\end{array}\right)\Sigma^{1/2}.
$$
\end{prop}
\begin{rem}
The proposition also follows if $X=Y$ by the remark right after Lemma~\ref{lem5}.
\end{rem}


\begin{proof}
By Lemma~\ref{lem5},  
$$
\mathbb{E}\exp\left\{t(Q-\mathbb{E}Q)\right\} \leq \exp\left\{\frac{1}{2}t^2 \tr(B^2)\right\}
$$
for $|t| < \frac{1}{2\|B\|_{op}}$.
Then 
$$
\mathbb{E}\exp(t\bar Q) \leq \exp\left\{ \frac{\tr(B^2)}{2n}t^2\right\}
$$
for $|t| <\frac{n}{2\|B\|_{op}} $. An application of Lemma~\ref{lem7} yields
$$
\mathbb{P}\left\{|\bar Q|>t \sqrt{\tr(B^2)}\right\} \leq 2\exp\left(- \frac{nt^2}{2}\right)
$$
for  $0 < t < \frac{\sqrt{\tr(B^2)}}{2\|B\|_{op}}$. It is easy to show that $\sqrt{\tr(B^2)}/\|B\|_{op}\geq 1$, hence we can always let $0 < t <1/2$.

\end{proof}

\begin{lem}
\label{lem4}
For $x>-1/2$,
$$
\log (1+x) - x + x^2 \geq 0.
$$
\end{lem}
\begin{proof}
Let $f(x) = \log (1+x) - x + x^2$. Then
$$\frac{\partial f(x)}{\partial x} = \frac{1}{1+x} - 1 + 2x = \frac{x(2x+1)}{x + 1}.
$$
So $\frac{\partial f(x)}{\partial x} <0$ if $-1/2 < x <0$,  $\frac{\partial f(x)}{\partial x} >0$ if $x>0$, and $f(0) = 0$, which leads to
$f(x)\geq 0$ for all $x>-1/2$.
\end{proof}

\begin{lem}
\label{lem5}
Let the $p\times 1$ vector $X$ and  $q\times 1$ vector $Y$ have a joint real normal 
distribution $\textnormal{MVN}_{p+q}(0,\Sigma)$, $\Sigma>0$. Let $A$ be a $p\times q$ real matrix. Let
$Q = X^TAY$. Then
\begin{equation}
\label{lem5_eq1}
\mathbb{E} \exp(tQ) = \frac{1}{\sqrt{\textnormal{det}\left(I_{p+q, p+q} - t B\right)}}
\end{equation}
for $|t| <\frac{1}{2\|B\|_{op}}$,
where $\textnormal{det}(\cdot)$ denotes the determinant of a square matrix and 
$$
B = \Sigma^{1/2}\left(\begin{array}{cc} 0_{p, p}& A\\
A^T&0_{q,q}\end{array}\right)\Sigma^{1/2}.
$$ 
Moreover,
\begin{equation}
\label{lem5_eq2}
\mathbb{E}\exp\left\{t(Q-\mathbb{E}Q)\right\} \leq \exp\left\{\frac{1}{2}t^2 \tr(B^2)\right\}
\end{equation}
for $|t| < \frac{1}{2\|B\|_{op}}$, where $\mathbb{E}Q = \tr(B)/2$.
\end{lem}
\begin{rem}
If $X=Y$, inequality (\ref{lem5_eq2}) still holds.
\end{rem}

\begin{proof}
\cite{Mathai:92} showed that 
$$
\mathbb{E} \exp(tQ) = \frac{1}{\sqrt{\textnormal{det}\left\{\Sigma^{-1} - t \left(\begin{array}{cc} 0_{p, p}& A\\
A^T&0_{q,q}\end{array}\right)\right\}\textnormal{det}(\Sigma) }}
$$
which leads to equation~(\ref{lem5_eq1}) by using the equality $\textnormal{det}(UV) = \textnormal{det}(U)\textnormal{det}(V)$ for square matrices $U$ and $V$ and that $\textnormal{det}(\Sigma) = \textnormal{det}(\Sigma^{1/2})\textnormal{det}(\Sigma^{1/2})$.
 Let $\{\lambda_1,\dots, \lambda_{p+q}\}$ denote the 
collection of all eigenvalues of $B$. Then
\begin{align*}
\mathbb{E}\exp(tQ) &= \prod_{k=1}^{p+q} (1- t\lambda_k)^{-1/2} \\
&= \exp\left\{-\frac{1}{2}\sum_{k=1}^{p+q} \log (1-t\lambda_k)\right\}\\
&\leq \exp\left\{-\frac{1}{2}\sum_{k=1}^{p+q}(-t\lambda_k - t^2\lambda_k^2) \right\}\\
&= \exp\left\{\frac{1}{2}t\sum_{k=1}^{p+q}\lambda_k +\frac{1}{2}t^2 \sum_{k=1}^{p+q} \lambda_k^2\right\}.
\end{align*}
In the above inequality we applied Lemma~\ref{lem4} for each $\log (1-t\lambda_k)$.
Note that 
\begin{align*}
\sum_{k=1}^{p+q} \lambda_k &= \tr(B)\\
& =\tr\left\{\left(\begin{array}{cc} 0_{p, p}& A\\
A^T&0_{q,q}\end{array}\right)\Sigma\right\}\\
& = \tr\left\{\left(\begin{array}{cc} 0_{p, p}& A\\
A^T&0_{q,q}\end{array}\right)\left(\begin{array}{cc}\Sigma_{11}&\Sigma_{12}\\
\Sigma_{21}&\Sigma_{22}\end{array}\right) \right\}\\
& = 2\tr(A\Sigma_{21})\\
& =2 \mathbb{E}Q
\end{align*}
and
$$
\sum_{k=1}^{p+q} \lambda_k^2 = \tr(B^2).
$$
Hence
$$
\mathbb{E}\exp\left\{t(Q-\mathbb{E}Q)\right\} \leq \exp\left\{\frac{1}{2}t^2  \tr(B^2)\right\}
$$
which is~(\ref{lem5_eq2}).
\end{proof}

\begin{lem}
\label{lem7}
Let $c_0$ and $c_1$ be two constants greater than 0.
Let $Q$ be a real random variable with mean zero and satisfies 
$$
\mathbb{E}\exp(tQ) \leq \exp(c_0t^2)
$$
for any $|t| < c_1$. Then for $0<t <2c_0c_1$,
$$
\mathbb{P}\left\{ Q\geq t\right\}\leq \exp\left(-\frac{t^2}{4c_0}\right)
$$
and
$$
\mathbb{P}\left\{ Q\leq - t\right\}\leq \exp\left(-\frac{t^2}{4c_0}\right).
$$
\end{lem}
\begin{proof}
Let $a = t/(2c_0)$. Then
\begin{align*}
\mathbb{P}(Q>t) &=\mathbb{P}\left\{\exp(aQ)>\exp(at)\right\}\\
&\leq \exp(-at)\mathbb{E} \exp(aQ)\\
&\leq \exp(c_0a^2 - at)\\
& =  \exp\left(-\frac{t^2}{4c_0}\right)
\end{align*}
Similarly we derive that 
$$
\mathbb{P}(Q<-t) \leq \exp\left(-\frac{t^2}{4c_0}\right)
$$
and the proof is complete.
\end{proof}

\begin{lem}
\label{lem6}
Let $A = (uv^T)*H$, where $u, v\in\mathbb{S}^{K-1}$ and $H = (h_{k\ell})_{1\leq k,\ell\leq K}$ with $|h_{k\ell}|\leq 1$ for all $(k,\ell)$. Let $\Sigma$ 
be an $2K\times 2K$ covariance matrix with 4 $K\times K$ blocks
$$
\Sigma = \left(\begin{array}{cc} \Sigma_{11}& \Sigma_{12}\\
\Sigma_{21}&\Sigma_{22}\end{array}\right).
$$
Let
$$
B = \Sigma^{1/2}\left(\begin{array}{cc} 0_{K, K}& A\\
A^T&0_{K,K}\end{array}\right)\Sigma^{1/2}.
$$ 
Then
$$
\tr(B^2) \leq 2\|\Sigma_{12}^{abs}\|_{op}^2 + 2\|\Sigma_{11}^{abs}\|_{op}\|\Sigma_{22}^{abs}\|_{op},
$$
where for a matrix $C= (c_{k\ell})_{1\leq k, \ell\leq K}$, $C^{abs}  = (|c_{k\ell}|)_{1\leq k, \ell\leq K}$.
\end{lem}
\begin{proof}
First we have
$
\tr(B^2) = \tr(C^2),
$
where 
\begin{align*}
C =& \left(\begin{array}{cc} 0_{K, K}& A\\
A^T&0_{K,K}\end{array}\right)\Sigma\\
=&\left(\begin{array}{cc} A\Sigma_{21}& A\Sigma_{22}\\
A^T\Sigma_{11}&A^T\Sigma_{12}\end{array}\right).
\end{align*}
Next
\begin{align*}
\tr(C^2) =&\tr(A\Sigma_{21}A\Sigma_{21} +A\Sigma_{22}A^T\Sigma_{11} + A^T\Sigma_{11}A\Sigma_{22}
+A^T\Sigma_{12}A^T\Sigma_{12} )\\
=&2\tr(A^T\Sigma_{12}A^T\Sigma_{12}) + 2\tr(A^T\Sigma_{11}A\Sigma_{22})\\
\leq&2\tr(A^{abs,T}\Sigma_{12}^{abs}A^{abs,T}\Sigma_{12}^{abs}) +  2\tr(A^{abs,T}\Sigma_{11}^{abs}A^{abs}\Sigma_{22}^{abs})\\
\leq&2\tr(v^{abs}u^{abs, T}\Sigma_{12}^{abs} v^{abs}u^{abs, T} \Sigma_{12}^{abs}) +  2\tr(v^{abs}u^{abs, T}\Sigma_{11}^{abs}u^{abs}v^{abs, T}\Sigma_{22}^{abs})\\
=&2(u^{abs, T}\Sigma_{12}v^{abs})^2 + 2(u^{abs, T}\Sigma_{11}^{abs}u^{abs}) (v^{abs, T}\Sigma_{22}^{abs}v^{abs})\\
\leq &2\|\Sigma_{12}^{abs}\|_{op}^2 + 2\|\Sigma_{11}^{abs}\|_{op}\|\Sigma_{22}^{abs}\|_{op}.
\end{align*}
\end{proof}

We put together  parts of the proof of Lemma 2 in \cite{Cai:12b} and have the following lemma.

\begin{lem}
\label{lem8}
For any matrix $A\in\mathbb{R}^{K\times K}$ and any $\delta$-net $\mathbb{Q}_K$ of $\mathbb{S}^{K-1}$ with $\delta <0.5$,
$$
\|A\|_{op} \leq (1-2\delta)^{-1} \sup_{u, v\in \mathbb{Q}_K} u^TAv.
$$
Moreover $\mathbb{Q}_K$ can be selected such that 
$$
\text{card}(\mathbb{Q}_K)\leq c\delta^{-K}K^{3/2}\log (1+K)
$$
for some absolute constant $c>0$. 
\end{lem}

\section*{Acknowledgement}
We thank Jacob Bien for kindly providing the code for the block thresholding estimator, many helpful suggestions and editing the paper.

\bibliographystyle{plain}
\bibliography{banding.bib}

\end{document}